\documentclass[12pt]{amsart}
\usepackage[margin=1.2in]{geometry}
\usepackage{graphicx,latexsym}
\usepackage{comment}
\usepackage{tikz}
\usepackage{mathrsfs}
\usepackage{amsfonts, amssymb, amsmath, amsthm, amsxtra, amscd, hyperref}
\usepackage{cite}
\usepackage{mathtools}
\usepackage{todonotes}
\usepackage[foot]{amsaddr}
\usepackage{comment}

\usepackage{tikz}
\usetikzlibrary{matrix,arrows}
\allowdisplaybreaks

\newcommand{\N}{\mathbb{N}}

\newcommand{\R}{\mathbb{R}}
\newcommand{\C}{\mathbb{C}}

\newcommand{\dx}{{\rm d}x }

\newcommand{\supp}{\operatorname{supp}}

\newcommand{\singsupp}{\operatorname{sing\,supp}\,}

\newtheorem{theorem}{Theorem}[section]
\newtheorem{proposition}[theorem]{Proposition}
\newtheorem{lemma}[theorem]{Lemma}
\newtheorem{corollary}[theorem]{Corollary}
\newtheorem{question}[theorem]{Question}

\theoremstyle{definition}

\newtheorem{example}[theorem]{Example}
\newtheorem{remark}[theorem]{Remark}
\newtheorem{problem}[theorem]{Problem}

\theoremstyle{remark}

\numberwithin{equation}{section}

\makeatletter
\DeclareRobustCommand\widecheck[1]{{\mathpalette\@widecheck{#1}}}
\def\@widecheck#1#2{%
    \setbox\z@\hbox{\m@th$#1#2$}%
    \setbox\tw@\hbox{\m@th$#1%
       \widehat{%
          \vrule\@width\z@\@height\ht\z@
          \vrule\@height\z@\@width\wd\z@}$}%
    \dp\tw@-\ht\z@
    \@tempdima\ht\z@ \advance\@tempdima2\ht\tw@ \divide\@tempdima\thr@@
    \setbox\tw@\hbox{%
       \raise\@tempdima\hbox{\scalebox{1}[-1]{\lower\@tempdima\box
\tw@}}}%
    {\ooalign{\box\tw@ \cr \box\z@}}}
\makeatother

\usetikzlibrary{arrows,chains,matrix,positioning,scopes}
\makeatletter
\tikzset{join/.code=\tikzset{after node path={%
\ifx\tikzchainprevious\pgfutil@empty\else(\tikzchainprevious)%
edge[every join]#1(\tikzchaincurrent)\fi}}}
\makeatother
\tikzset{>=stealth',every on chain/.append style={join},
         every join/.style={->}}
\tikzstyle{labeled}=[execute at begin node=$\scriptstyle,
   execute at end node=$]

\begin{document}
	
\title[Linear topological invariants for kernels of convolution operators]{Linear topological invariants for kernels of convolution and  differential operators}
	
\author[A. Debrouwere]{A. Debrouwere$^1$}
\address{$^1$Department of Mathematics and Data Science, Vrije Universiteit Brussel, Pleinlaan 2, 1050 Brussels, Belgium}
\email{andreas.debrouwere@vub.be}

\author[T.\ Kalmes]{T.\ Kalmes$^2$}
\address{$^2$Faculty of Mathematics, Chemnitz University of Technology, 09107 Chemnitz, Germany}
\email{thomas.kalmes@math.tu-chemnitz.de}

\begin{abstract}
	We establish the condition ($\Omega$) for smooth kernels of various types of convolution and differential operators. By the  (DN)-$(\Omega)$ splitting theorem of Vogt and Wagner, this implies that these operators are surjective on the corresponding spaces of vector-valued smooth functions with values in a product of Montel (DF)-spaces whose strong duals satisfy the condition (DN), e.g., the space $\mathscr{D}'(Y)$ of distributions over an open set $Y \subseteq \R^{n}$ or the space $\mathscr{S}'(\R^{n})$ of tempered distributions. Most notably, we show that:
	\begin{itemize}
		\item[$(i)$] $\mathscr{E}_P(X) =  \{ f \in \mathscr{E}(X) \, | \, P(D)f = 0 \}$  satisfies ($\Omega$) for any differential operator $P(D)$ and any open convex set $X \subseteq \R^d$.
		\item[$(ii)$] Let $P\in\C[\xi_1,\xi_2]$ and $X \subseteq \R^2$ open be such that $P(D):\mathscr{E}(X)\rightarrow\mathscr{E}(X)$ is surjective. Then, $\mathscr{E}_P(X)$ satisfies ($\Omega$).
		\item[$(iii)$] Let $\mu \in \mathscr{E}'(\R^d)$ be such that $ \mathscr{E}(\R^d) \rightarrow \mathscr{E}(\R^d), \, f \mapsto \mu \ast f$  is surjective. Then, $ \{ f \in \mathscr{E}(\R^d) \, | \, \mu \ast f = 0 \}$  satisfies ($\Omega$).
	\end{itemize}
	The central result in this paper is that the space of smooth zero solutions  of a general convolution equation satisfies the condition ($\Omega$) if and only if  the space of distributional zero solutions of the equation satisfies the condition (P$\Omega$). The above and related statements then follow from known results concerning (P$\Omega$)  for distributional kernels of convolution and differential operators \cite{B-D2006, Kalmes12-3,Kalmes19}.

\mbox{}\\
	
	
	
\end{abstract}

\maketitle

\section{Introduction}
The aim of this paper is to study certain linear topological invariants for kernels of convolution and differential operators. This topic goes back to the seminal article of Vogt \cite{Vogt1983-2} and has gained a renewed interest by the work of Bonet and Doma\'nski \cite{B-D2006,B-D2008}. As we shall explain later on, this problem is closely connected to and motivated by the question of surjectivity of such operators on spaces of vector-valued smooth functions and distributions: 
\begin{question}\label{question-intro}Let $\mu \in\mathscr{E}'(\R^d)$ and $X_1, X_2 \subseteq \R^d$ be open sets such that 
	\begin{equation}
		\label{def-sets}X_2 - \operatorname{supp} \mu \subseteq X_1.
	\end{equation} 
	Let $E$ be a locally convex space.
	\begin{itemize}
		\item[$(i)$] Suppose that $\mathscr{E}(X_1) \to \mathscr{E}(X_2), \, f \mapsto \mu \ast f$ is surjective. When is the associated map
		\begin{equation}
			\mathscr{E}(X_1;E) \to \mathscr{E}(X_2;E), \, {\mathbf{f}} \mapsto \mu \ast {\mathbf{f}} 
			\label{vv-smooth}
		\end{equation}
		surjective?
		\item[$(ii)$] Suppose that $\mathscr{D}'(X_1) \to \mathscr{D}'(X_2), \, f \mapsto \mu \ast f$ is surjective. When is the associated map
		\begin{equation} 
			\mathscr{D}'(X_1;E) \to \mathscr{D}'(X_2;E),  \, {\mathbf{f}} \mapsto \mu \ast {\mathbf{f}} 
			\label{vv-dist}
		\end{equation}
		surjective?
	\end{itemize}
\end{question}
\noindent If $E$ is a space of functions or distributions, Question \ref{question-intro} is equivalent to the problem of parameter dependence of solutions of convolution equations. For differential equations this problem  has a long and rich tradition, see \cite{B-D2006,B-D2008, Browder, Domanski, Mantlik1, Mantlik2, Mantlik3, Treves1, Treves2}. 

We now explain the connection between Question \ref{question-intro} and linear topological invariants for kernels of convolution operators. We assume that the reader is familiar with the conditions ($\Omega$) and (DN) for Fr\'echet spaces and the condition (P$\Omega$) for (PLS)-spaces. We refer to \cite{B-D2006,M-V} for more information on these conditions and examples of spaces satisfying them. Suppose that $\mathscr{E}(X_1) \to \mathscr{E}(X_2), \, f \mapsto \mu \ast f$ is surjective and denote by $\mathscr{E}_\mu(X_1,X_2)$ the kernel of this map. A result of  Grothendieck \cite{Grothendieck} yields that the map \eqref{vv-smooth} is always surjective if $E$ is a Fr\'echet space.  This is no longer true in general if $E$ is a (DF)-space, as shown by Vogt \cite{Vogt1983-2}.  The splitting theory for Fr\'echet spaces \cite{Vogt1987} implies that for $E =s'$, with $s$ the  space of rapidly decreasing sequences, the map \eqref{vv-smooth} is surjective if and only if $\mathscr{E}_\mu(X_1,X_2)$ satisfies ($\Omega$) (cf. \cite{Vogt1983-2}). Furthermore, if  $\mathscr{E}_\mu(X_1,X_2)$ satisfies ($\Omega$), then the map \eqref{vv-smooth} is surjective  for any $E$ isomorphic to  a product of Montel (DF)-spaces whose strong duals satisfy (DN). Counterparts of these results for distributions were obtained by Bonet and Doma\'nski \cite{B-D2006}. Suppose that $\mathscr{D}'(X_1) \to \mathscr{D}'(X_2), \, f \mapsto \mu \ast f$ is surjective and  denote by $\mathscr{D}'_\mu(X_1,X_2)$ the kernel of this map.
For $E = \mathscr{D}'(\R) \cong \prod_{n \in \N} s'$  the map \eqref{vv-dist} is surjective if and only if $\mathscr{D}'_\mu(X_1,X_2)$ satisfies (P$\Omega$).  Moreover, if $\mathscr{D}'_\mu(X_1,X_2)$ satisfies (P$\Omega$), then the map \eqref{vv-dist} is surjective for any  $E$ isomorphic to a product of  (DFN)-spaces  whose strong duals satisfy (DN).

In view of the above results, it is natural to study when the spaces $\mathscr{E}_\mu(X_1,X_2)$ and $\mathscr{D}'_\mu(X_1,X_2)$ satisfy ($\Omega$) and (P$\Omega$), respectively. This problem has been mainly considered for differential operators. 
Let $P \in \C[\xi_1, \ldots, \xi_d]$ be a polynomial and consider the corresponding  differential operator $P(D)=P(-i\frac{\partial}{\partial x_1},\ldots,-i\frac{\partial}{\partial x_d})$. For an open set $X \subseteq \R^d$  we write $\mathscr{D}'_P(X) = \{ f \in \mathscr{D}'(X) \, | \, P(D)f = 0 \}$ and $\mathscr{E}_P(X) = \mathscr{D}'_P(X) \cap \mathscr{E}(X)$. Suppose  that $P(D):\mathscr{D}'(X)\rightarrow\mathscr{D}'(X)$ is surjective. As mentioned above, $\mathscr{D}'_P(X)$ satisfies (P$\Omega$) if and only if $P(D)   :  \mathscr{D}'(X;\mathscr{D}'(\R)) \to \mathscr{D}'(X;\mathscr{D}'(\R))$ is surjective. Hence, the Schwartz kernel theorem yields that $\mathscr{D}'_P(X)$ satisfies (P$\Omega$) if and only if  $P^+(D) :  \mathscr{D}'(X \times \R) \to \mathscr{D}'(X \times \R)$ is surjective, where $P^+(\xi_1, \ldots, \xi_{d+1}) = P(\xi_1,\ldots, \xi_d)$  \cite{B-D2006}. Combining this characterization  with classical results of H\"ormander about the surjectivity of differential operators  on distribution spaces \cite{HoermanderPDO2} (applied to $P^+(D)$), gives a powerful method to study (P$\Omega$) for $\mathscr{D}'_P(X)$.
For instance, it gives directly that  $\mathscr{D}'_P(X)$ always satisfies (P$\Omega$) if  $X$ is convex \cite{B-D2006} and it may be used to show that  $\mathscr{D}'_P(X)$ satisfies (P$\Omega$) if $P(D)$ is elliptic and $X$ is an arbitrary open set \cite{FrKa, Vogt1983-2}. In \cite{Kalmes12-3, Kalmes19} the second author used this method to  show that $\mathscr{D}'_P(X)$ satisfies (P$\Omega$) whenever $P(D)$ is surjective on $\mathscr{D}'(X)$ in the following cases: $d = 2$; $P(D)$ is semi-elliptic with a single characteristic direction; $P(D)$ acts along a subspace and is elliptic on this subspace; $P(D)$ is a product of first order operators. We remark that in all the above cases also  a geometric characterization of the sets $X$ for which  $P(D):\mathscr{D}'(X)\rightarrow\mathscr{D}'(X)$ is surjective is known \cite{HoermanderPDO2,Kalmes11, Kalmes19}. On the negative side, in \cite{Kalmes12-2} this method was used to provide  a concrete example of a (hypoelliptic)  surjective operator $P(D)$ on $\mathscr{D}'(X)$ such that $\mathscr{D}'_P(X)$ does not satisfy (P$\Omega$).

There does not seem to exist an analogue of the above approach to study ($\Omega$) for $\mathscr{E}_P(X)$. However, if $P(D)$ is hypoelliptic, then $\mathscr{E}_P(X) = \mathscr{D}'_P(X)$ as locally convex spaces. Since  for a (FS)-space the conditions ($\Omega$) and (P$\Omega$) are equivalent, we obtain that  $\mathscr{E}_P(X)$ satisfies ($\Omega$) if and only if $\mathscr{D}'_P(X)$ satisfies (P$\Omega$) in the hypoelliptic case. The above mentioned results concerning (P$\Omega$) for $\mathscr{D}'_P(X)$ may therefore be transferred to ($\Omega$) for $\mathscr{E}_P(X)$ if $P(D)$ is hypoelliptic, e.g., $\mathscr{E}_P(X)$  satisfies ($\Omega$) if  $P(D)$ is hypoelliptic and $X$ is convex or if $P(D)$ is elliptic and $X$ is arbitrary. These results were obtained previously by Petzsche \cite{Petzsche1980} and Vogt \cite{Vogt1983-2}, respectively, via different more direct methods. Note that the example from \cite{Kalmes12-2} also shows that for surjective operators $P(D)$ on $\mathscr{E}(X)$ the smooth kernel $\mathscr{E}_P(X)$ in general does not satisfy ($\Omega$).

The principal goal of this paper is to  establish ($\Omega$) for smooth kernels of (non-hypoelliptic) differential operators and, more generally, convolution operators. Our main tool is the following result, which extends the equivalence of ($\Omega$) for $\mathscr{E}_P(X)$ and (P$\Omega$) for $\mathscr{D}'_P(X)$ in the hypoelliptic case to general convolution operators:
\begin{theorem}\label{main-intro} Let  $\mu \in \mathscr{E}'(\R^d)$ and $X_1, X_2 \subseteq \R^d$ be open sets such that \eqref{def-sets} holds.
	\begin{itemize}
		\item[$(i)$] Suppose that $ \mathscr{E}(X_1) \rightarrow \mathscr{E}(X_2), \, f \mapsto \mu \ast f$  is surjective. If $\mathscr{D}'_\mu (X_1,X_2)$ satisfies (P$\Omega$), then $\mathscr{E}_\mu (X_1,X_2)$ satisfies ($\Omega$).
		\item[$(ii)$] Suppose that  $\mathscr{D}'(X_1) \rightarrow \mathscr{D}'(X_2), \, f \mapsto \mu \ast f$ is surjective. If  $\mathscr{E}_\mu (X_1,X_2)$ satisfies ($\Omega$), then $\mathscr{D}'_\mu (X_1,X_2)$ satisfies (P$\Omega$).
	\end{itemize}
\end{theorem}
\noindent Theorem \ref{main-intro} should be compared with a classical result of Meise, Taylor and Vogt \cite{MTV} stating that a differential operator $P(D)$ admits a continuous linear right inverse on $\mathscr{E}(X)$,  $X \subseteq \R^d$ open, if and only if it does so on $\mathscr{D}'(X)$. 

The proof of Theorem \ref{main-intro} is given in Section \ref{sect-main} and is based on some abstract results about the derived projective limit functor \cite{Wengenroth} and a simple regularization procedure. The necessary notions and results from the theory of  the derived projective limit functor are discussed in the preliminary Section \ref{sect-prelim}.

By combining Theorem \ref{main-intro} with results concerning (P$\Omega$) for distributional kernels of various types of convolution and differential operators, we show ($\Omega$) for smooth kernels of these operators in Section \ref{sect-examples}. Most notably, we prove that:
\begin{itemize}
	\item[$(i)$] $\mathscr{E}_P(X)$  satisfies ($\Omega$) for any differential operator $P(D)$ and any open convex set $X \subseteq \R^d$ (Theorem \ref{theo: convex}).
	\item[$(ii)$] Let $P\in\C[\xi_1,\xi_2]$ and $X \subseteq \R^2$ open be such that $P(D):\mathscr{E}(X)\rightarrow\mathscr{E}(X)$ is surjective. Then, $\mathscr{E}_P(X)$ satisfies ($\Omega$) (Theorem \ref{theo-main-examples}$(iii)$).
	\item[$(iii)$] Let $\mu \in \mathscr{E}'(\R^d)$ be such that $ \mathscr{E}(\R^d) \rightarrow \mathscr{E}(\R^d), \, f \mapsto \mu \ast f$  is surjective. Then, $\mathscr{E}_\mu(\R^d,\R^d)$  satisfies ($\Omega$) (Theorem \ref{theo: convex-conv}).
	\item[$(iv)$] Let $\mu \in \mathscr{E}'(\R^d)$ belong to the class $\mathscr{R}$ of Berenstein and Dostal \cite{Be-Do1973-1} (see also \cite{Be-Do1973,Dostal}). Let   $X \subseteq \R^d$ be open and convex, and set $X_\mu = \{ x \in \R^d \, | \, x - \supp \mu \subseteq X \}$. Then, $\mathscr{D}'_\mu(X,X_\mu)$  satisfies (P$\Omega$) and $\mathscr{E}_\mu(X_,X_\mu)$  satisfies ($\Omega$) (Theorem \ref{theo:conv-conv}). 
\end{itemize} 
We mention that every distribution with  finite support is of class $\mathscr{R}$ \cite{Be-Do1973-1}. Hence, $(iv)$ particularly applies to difference-differential operators as they are precisely the convolution operators whose kernels have finite support (Corollary \ref{diff-diff}). To the best of our knowledge, $(iv)$ is the first result about linear topological invariants for kernels of convolution operators on convex sets other than the whole space $\R^d$. 

\section{Preliminaries}\label{sect-prelim}
In this preliminary section we recall various notions and results concerning the derived projective limit functor \cite{Wengenroth}, the linear topological invariants ($\Omega$) and (P$\Omega$) \cite{M-V,B-D2006}, and convolution operators \cite{HoermanderPDO2}.

Throughout this article we use standard notation from functional analysis \cite{M-V, Wengenroth} and distribution theory \cite{Schwartz, HoermanderPDO1, HoermanderPDO2}. In particular, given a locally convex space $X$,  we denote by $\mathscr{U}_0(X)$ the filter basis of absolutely convex  neighborhoods of $0$ in $X$ and by $\mathscr{B}(X)$ the family of all absolutely convex bounded sets in $X$.

\subsection{Projective spectra}
A \emph{projective spectrum}  $\mathscr{X} = (X_n, \varrho_{n+1}^n)_{n \in \N}$ consists of vector spaces $X_n$  and  linear 
maps $\varrho^n_{n+1}: X_{n+1} \rightarrow X_n$, called the \emph{spectral maps}. We define $\varrho^n_n  = \operatorname{id}_{X_n}$ and $\varrho^n_{m} = \varrho_{n+1}^n \circ \cdots \circ \varrho^{m-1}_{m} : X_m \rightarrow X_n$ for $n,m \in \N$ with $m > n$.  The \emph{projective limit} of $\mathscr{X}$ is defined as
\[\operatorname{Proj} \mathscr{X} = \left\{(x_n)_{n \in \N} \in \prod_{n \in \N} X_n \, | \,  x_n = \varrho^n_{n+1}(x_{n+1}),  \forall n \in \N\right\}.\]
For each $n \in \N$ we  write $\varrho^n : \operatorname{Proj} \mathscr{X} \rightarrow X_n, \, (x_j)_{j \in \N} \mapsto x_n$. 

Given two projective spectra $\mathscr{X} = (X_n, \varrho_{n+1}^n)_{n \in \N}$  and $\mathscr{Y} = (Y_n, \sigma_{n+1}^n)_{n \in \N}$, a \emph{morphism} $S= (S_n)_{n \in \N}: \mathscr{X} \rightarrow \mathscr{Y}
$ consists of  linear maps $S_n : X_n \rightarrow Y_n$ such that $S_n \circ \varrho_{n+1}^n = \sigma_{n+1}^n \circ S_{n+1}$ for all $n \in \N$. We define $\ker S$ as the projective spectrum  $( \ker S_n, \varrho^n_{n+1 | \ker S_{n+1}})_{n \in \N}$ and
\[\operatorname{Proj} S : \operatorname{Proj} \mathscr{X} \rightarrow \operatorname{Proj} \mathscr{Y}, \, (x_n)_{n \in \N} \mapsto (S_n(x_n))_{n \in \N}.\]

The \emph{derived projective limit} of a projective spectrum $\mathscr{X} = (X_n, \varrho_{n+1}^n)_{n \in \N}$ is defined as
\[\operatorname{Proj}^1 \mathscr{X} =   \prod_{n \in \N}X_n / B(\mathscr{X}),\]
where
\begin{equation*}
	\begin{gathered}
		B(\mathscr{X}) = \{(x_n)_{n \in \N} \in \prod_{n \in \N} X_n \, | \, \exists \, (u_n)_{n \in \N} \in \prod_{n \in \N} X_n \, : \\  x_n = u_n - \varrho^n_{n+1}(u_{n+1}), \forall n \in \N\}.
	\end{gathered}
\end{equation*}
We shall use the following fundamental property of the derived projective limit.
\begin{proposition} \label{proj1surj}\cite[Theorem 3.1.8]{Wengenroth} 
	Let $\mathscr{X} = (X_n, \varrho_{n+1}^n)_{n \in \N}$ and $\mathscr{Y} = (Y_n, \sigma_{n+1}^n)_{n \in \N}$ be projective spectra and let $S = (S_n)_{n \in \N}: \mathscr{X} \rightarrow \mathscr{Y}$ be a morphism. Suppose that for all $n \in \N$ there is $m \geq n$ such that $\sigma^n_m(Y_m) \subseteq S_n(X_n)$. Then, there is an exact sequence of linear maps
	\begin{center}
		\begin{tikzpicture}
			\matrix (m) [matrix of math nodes, row sep=2em, column sep=1.2em]
			{0 & \operatorname{Proj} \ker S & \operatorname{Proj}\mathscr{X} & \operatorname{Proj}\mathscr{Y} & \operatorname{Proj}^1 \ker S & \operatorname{Proj}^1 \mathscr{X}. \\
			};
			{ [start chain] \chainin (m-1-1);
				\chainin (m-1-2);
				\chainin (m-1-3) [join={node[above,labeled] {}}];
				\chainin (m-1-4)[join={node[above,labeled] {\operatorname{Proj} S }}];
				\chainin (m-1-5)[join={node[above,labeled] {}}];
				\chainin (m-1-6)[join={node[above,labeled] {}}];
			}
			
		\end{tikzpicture}
	\end{center}
	Consequently, $\operatorname{Proj}^1 \ker S = 0$ implies that $\operatorname{Proj} S : \operatorname{Proj}\mathscr{X} \rightarrow \operatorname{Proj}\mathscr{Y}$ is surjective. If  $\operatorname{Proj}^1\mathscr{X}= 0$, the converse holds true as well. 
\end{proposition}

By a \emph{projective spectrum of locally convex spaces} we mean a projective spectrum  $\mathscr{X} = (X_n, \varrho_{n+1}^n)_{n \in \N}$ consisting of locally convex spaces $X_n$  and continuous linear spectral maps $\varrho^n_{n+1}$.
In such a case, we endow $\operatorname{Proj} \mathscr{X}$ with its natural projective limit topology. 

For projective spectra of Fr\'echet spaces  the vanishing of the derived projective limit may be characterized as follows (this result is known as the Mittag-Leffler lemma).

\begin{theorem}\label{M-L} \cite[Theorem 3.2.8]{Wengenroth} Let $\mathscr{X} = (X_n, \varrho_{n+1}^n)_{n \in \N}$ be a projective spectrum of Fr\'echet spaces. Then, the following statements are equivalent:
	\begin{itemize}
		\item[$(i)$] $\operatorname{Proj}^1 \mathscr{X} = 0$.
		\item[$(ii)$] $\forall n \in \N, \, U \in \mathscr{U}_0(X_n) \, \exists m \geq n \, \forall k \geq m \, : \, \varrho^n_m(X_m) \subseteq \varrho^n_k(X_k) + U$.
		\item[$(iii)$] $\forall n \in \N, \, U \in \mathscr{U}_0(X_n) \, \exists m \geq n  \, : \, \varrho^n_m(X_m) \subseteq \varrho^n(\operatorname{Proj} \mathscr{X}) + U$. 
	\end{itemize} 
\end{theorem}
For  projective spectra of (DFS)-spaces the following characterization holds.
\begin{theorem}\label{M-L-DFS} \cite[Theorem 3.2.18]{Wengenroth} Let $\mathscr{X} = (X_n, \varrho_{n+1}^n)_{n \in \N}$ be a projective spectrum of (DFS)-spaces. Then, $\operatorname{Proj}^1 \mathscr{X} = 0$ if and only if
	\[\begin{gathered}
		\forall n \in \N \, \exists m \geq n \, \forall k \geq m \, \exists B \in \mathscr{B}(X_n) \, \forall M \in \mathscr{B}(X_m) \, \exists K \in \mathscr{B}(X_k) \,  : \\
		\varrho^n_m(M) \subseteq  B +  \varrho^n_k(K).
	\end{gathered}\]
\end{theorem}

\subsection{The condition ($\Omega$) for Fr\'echet spaces}
A Fr\'echet space $X$ is said to satisfy the condition ($\Omega$) \cite{M-V} if
\[\begin{gathered}
	\forall U \in \mathscr{U}_0(X) \, \exists V \in \mathscr{U}_0(X) \, \forall W \in \mathscr{U}_0(X) \, \exists C,s > 0 \, \forall \varepsilon \in (0,1) \, : \\ V \subseteq \varepsilon U + \frac{C}{\varepsilon^s} W.
\end{gathered}\]
The following result concerning ($\Omega$) for projective limits of spectra of Fr\'echet spaces is sometimes useful for verifying this condition in concrete situations. For us it will play a vital role in the next section.  We believe this result is essentially known, but  we include a proof here  as we could not find one in the literature.

\begin{lemma}\label{Omega-step}
	Let $\mathscr{X} = (X_n, \varrho^n_{n+1})_{n \in \N}$ be a projective spectrum of Fr\'echet spaces. Then,  $\operatorname{Proj} \mathscr{X}$ satisfies ($\Omega$) and $\operatorname{Proj}^1 \mathscr{X}  =0$ if and only if
	\begin{eqnarray}\label{Omega-step-form}
		\begin{gathered}
			\forall n \in \N, U \in \mathscr{U}_0(X_n) \, \exists m \geq n, V \in \mathscr{U}_0(X_m) \, \forall k \geq m, W \in \mathscr{U}_0(X_k) \\ \exists C,s > 0 \, \forall \varepsilon \in (0,1)\, : \, 
			\varrho^n_m(V) \subseteq \varepsilon U + \frac{C}{\varepsilon^s} \varrho^n_k(W).
		\end{gathered}		
	\end{eqnarray}
\end{lemma}
\begin{proof} By definition of the projective limit topology, ($\Omega$) for $X = \operatorname{Proj} \mathscr{X}$ means that
	\begin{eqnarray}\label{Omegaa}
		\begin{gathered}
			\forall n \in \N, U \in \mathscr{U}_0(X_n) \, \exists m \geq n, V \in \mathscr{U}_0(X_m) \, \forall k \geq m, W \in \mathscr{U}_0(X_k) \\ \exists C,s > 0 \, \forall \varepsilon \in (0,1)\, :\,
			(\varrho^m)^{-1}(V) \subseteq \varepsilon  (\varrho^n)^{-1}(U) + \frac{C}{\varepsilon^s} (\varrho^k)^{-1}(W). 
		\end{gathered}
	\end{eqnarray}
	Now suppose that  \eqref{Omega-step-form}  holds. This condition clearly implies condition $(ii)$ from Theorem \ref{M-L} and thus that $\operatorname{Proj}^1 \mathscr{X}  =0$. We now show \eqref{Omegaa}. Let $n \in \N$ and $U \in \mathscr{U}_0(X_n)$ be arbitrary and choose $m \geq n$  and $V \in \mathscr{U}_0(X_m)$ according to \eqref{Omega-step-form}. Let $k \geq m$ and $W \in \mathscr{U}_0(X_k)$ be arbitrary. By condition $(iii)$ from Theorem \ref{M-L}, there is $\widetilde{k} \geq k$ such that
	\begin{equation}
		\label{Proj1-cons0}
		\varrho^{k}_{\widetilde{k}}(X_{\widetilde{k}}) \subseteq \varrho^{k}(X) + (W \cap (\varrho^n_k)^{-1}(U) ).
	\end{equation}
	Set $\widetilde{W} = (\varrho^{k}_{\widetilde{k}})^{-1}(W)$. Condition \eqref{Omega-step-form} (applied to $\widetilde{k}$ and $\widetilde{W}$ instead of $k$ and $W$)  yields that there are $C,s > 0$ such that for all $\varepsilon \in (0,1)$
	\begin{equation}\label{Omega-proj-20}
		\varrho^n_m(V) \subseteq \varepsilon U + \frac{C}{\varepsilon^s} \varrho^n_{\widetilde{k}}(\widetilde{W}). 
	\end{equation}
	Let $x \in  (\varrho^m)^{-1}(V)$ and $\varepsilon \in (0,1)$ be arbitrary.  Note that $\varrho^n(x) = \varrho^n_m(\varrho^m(x)) \in \varrho^n_m(V)$. Condition \eqref{Omega-proj-20} implies that there is $\widetilde{y}_\varepsilon \in \frac{2^sC}{\varepsilon^s} \widetilde{W}$ such that $\varrho^n(x) - \varrho^n_{\widetilde{k}}(\widetilde{y}_\varepsilon) \in \frac{\varepsilon}{2}U$. By multiplying both sides of \eqref{Proj1-cons0} with $\varepsilon/2$, we  find that there is $y_\varepsilon \in X$ such that $\varrho^k_{\widetilde{k}}(\widetilde{y}_\varepsilon) - \varrho^k(y_\varepsilon) \in \frac{\varepsilon}{2}(W \cap (\varrho^n_k)^{-1}(U)) \subseteq \frac{1}{\varepsilon^s}W \cap \frac{\varepsilon}{2}(\varrho^n_k)^{-1}(U)$. We have that
	\begin{eqnarray*}
		\varrho^n(x-y_\varepsilon) &=& (\varrho^n(x) - \varrho^n_{\widetilde{k}}(\widetilde{y}_\varepsilon)) + (\varrho^n_{\widetilde{k}}(\widetilde{y}_\varepsilon) - \varrho^n(y_\varepsilon)) \\
		&=&  (\varrho^n(x) - \varrho^n_{\widetilde{k}}(\widetilde{y}_\varepsilon)) + \varrho^n_{k}(\varrho^k_{\widetilde{k}}(\widetilde{y}_\varepsilon) - \varrho^k(y_\varepsilon)) \in \frac{\varepsilon}{2} U + \frac{\varepsilon}{2} U = \varepsilon U
	\end{eqnarray*}
	and
	\[\varrho^k(y_\varepsilon) = (\varrho^k(y_\varepsilon)-\varrho^k_{\widetilde{k}}(\widetilde{y}_\varepsilon)) + \varrho^k_{\widetilde{k}}(\widetilde{y}_\varepsilon) \in \frac{1}{\varepsilon^s}W +  \frac{2^sC}{\varepsilon^s}W = \frac{2^sC+1}{\varepsilon^s}W.\]
	Thus,
	\[x = (x-y_\varepsilon) + y_\varepsilon \in \varepsilon  (\varrho^n)^{-1}(U) + \frac{2^sC+1}{\varepsilon^s} (\varrho^k)^{-1}(W).\]
	Next, assume that  $X$ satisfies ($\Omega$) (and thus  \eqref{Omegaa}) and $\operatorname{Proj}^1 \mathscr{X} = 0$. Let $n \in \N$ and $U \in \mathscr{U}_0(X_n)$ be arbitrary.  Choose $\widetilde{m} \geq n$ and $\widetilde{V} \in \mathscr{U}_0(X_{\widetilde{m}})$ according to \eqref{Omegaa}. Since  $\operatorname{Proj}^1 \mathscr{X} = 0$, Theorem \ref{M-L} implies that there is $m \geq \widetilde{m}$ such that
	\begin{equation}
		\label{Proj1-cons}
		\varrho^{\widetilde{m}}_m(X_m) \subseteq \varrho^{\widetilde{m}}(X) + (\widetilde{V} \cap (\varrho^n_{\widetilde{m}})^{-1}(U)).
	\end{equation}
	Set $V = \frac{1}{2} (\varrho^{\widetilde{m}}_m)^{-1}(\widetilde{V})$. Let  $k \geq m$ and $W \in \mathscr{U}_0(X_k)$ be arbitrary. Condition \eqref{Omegaa} implies that there are $C,s > 0$ such that for all $\varepsilon \in (0,1)$
	\begin{equation}
		\label{Omega-proj-2}
		(\varrho^{\widetilde{m}})^{-1}(\widetilde{V}) \subseteq \varepsilon  (\varrho^n)^{-1}(U) + \frac{C}{\varepsilon^s} (\varrho^k)^{-1}(W). 
	\end{equation}
	Let $x \in V$ and $ \varepsilon \in (0,1)$ be arbitrary. By multiplying both sides of \eqref{Proj1-cons} with $\varepsilon/2$, we  find that there are $y_\varepsilon \in X$ and $z_\varepsilon \in \frac{\varepsilon}{2}(\widetilde{V} \cap (\varrho^n_{\widetilde{m}})^{-1}(U)) \subseteq \frac{1}{2}\widetilde{V} \cap \frac{\varepsilon}{2}(\varrho^n_{\widetilde{m}})^{-1}(U)$ such that 
	$\varrho^{\widetilde{m}}_m(x) = \varrho^{\widetilde{m}}(y_\varepsilon) + z_\varepsilon$. Note that $y_\varepsilon \in  (\varrho^{\widetilde{m}})^{-1}(\widetilde{V})$. Hence, by \eqref{Omega-proj-2},
	\[\varrho^{n}_m(x) = \varrho^{n}(y_\varepsilon) +\varrho^{n}_{\widetilde{m}}(z_\varepsilon) \in \frac{\varepsilon}{2}U + \frac{2^sC}{\varepsilon^s} \varrho^n((\varrho^k)^{-1}(W))  + \frac{\varepsilon}{2}U \subseteq \varepsilon U + \frac{2^sC}{\varepsilon^s} \varrho^n_k(W).\]
\end{proof}

\subsection{The condition (P$\Omega$) for (PLS)-spaces}
A locally convex space $X$ is called a (PLS)-space if it can be written as the topological projective limit of a spectrum of (DFS)-spaces. 

Let  $\mathscr{X} = (X_n, \varrho^n_{n+1})_{n \in \N}$ be a spectrum of (DFS)-spaces. We call $\mathscr{X}$ \emph{strongly reduced} if
\[\forall n \in \N \, \exists m \geq n \, : \, \varrho^n_m(X_m) \subseteq \overline{\varrho^n( \operatorname{Proj} \mathscr{X} )}^{X_n}.\] 
The spectrum $\mathscr{X}$ is said to satisfy (P$\Omega$) if
\begin{eqnarray*}
	\begin{gathered}
		\forall n \in \N \, \exists m \geq n \, \forall k \geq m \, \exists B \in \mathscr{B}(X_n) \, \forall M \in \mathscr{B}(X_m) \, \exists K \in \mathscr{B}(X_k) \, \exists C,s > 0 \\ \forall \varepsilon \in (0,1) \, : \, 
		\varrho^n_m(M) \subseteq \varepsilon B + \frac{C}{\varepsilon^s} \varrho^n_k(K).
	\end{gathered}
\end{eqnarray*}

A (PLS)-space $X$ is said to satisfy (P$\Omega$) if  $X = \operatorname{Proj}\mathscr{X}$ for some strongly reduced spectrum $\mathscr{X}$ of (DFS)-spaces that satisfies (P$\Omega$). This notion is well-defined as \cite[Proposition 3.3.8]{Wengenroth} yields that all strongly reduced projective spectra $\mathscr{X}$ of (DFS)-spaces with $X = \operatorname{Proj}\mathscr{X}$ are equivalent (in the sense of  \cite[Definition 3.1.6]{Wengenroth}). The bipolar theorem and \cite[Lemma 4.5]{B-D2006} imply that the above definition of (P$\Omega$) is equivalent to the original one from \cite{B-D2006}.
\subsection{Convolution operators} 
Let  $\mu \in \mathscr{E}'(\R^d)$. For all open sets $X_1, X_2 \subseteq \R^d$  such that \eqref{def-sets} holds the convolution operators
\begin{equation}
	S_\mu  : \mathscr{D}'(X_1) \rightarrow \mathscr{D}'(X_2), \, f \mapsto \mu \ast f 
	\label{map-1}
\end{equation}
and
\begin{equation}
	S_\mu : \mathscr{E}(X_1) \rightarrow \mathscr{E}(X_2)
	\label{map-2}
\end{equation}
are well-defined continuous linear maps. H\"ormander characterized when the maps \eqref{map-1} and \eqref{map-2} are surjective \cite[Section 16.5]{HoermanderPDO2}. We need some preparation to state his results. As customary, we define the Fourier transform of an element $\mu \in \mathscr{E}'(\R^d)$ as
\[\widehat{\mu}(\zeta) = \langle \mu(x), e^{-i x \cdot \zeta} \rangle, \qquad \zeta \in \C^d.\]
Then, $\widehat{\mu}$ is an entire function such that
\[|\widehat{\mu}(\zeta)| \leq C(1+|\zeta|)^N e^{H_{\mu}(\operatorname{Im} \zeta)}, \qquad \zeta \in \C^d,\]
for some $C,N > 0$, where $H_\mu$ denotes the supporting function of $\supp \mu$.  A distribution $\mu \in \mathscr{E}'(\R^d)$ is called \emph{invertible} \cite[Definition 16.3.12]{HoermanderPDO2} (see also \cite{Ehrenpreis}) if 
\[\exists c,R,M > 0 \, \forall \xi \in \R^d \, \exists \zeta \in \C^d, |\zeta - \xi| < R \log(1+|\xi|) \, : \,  |\widehat{\mu}(\zeta)| \geq c(1+|\xi|)^{-M}.\]
We refer to \cite[Theorem 16.3.10]{HoermanderPDO2} for various characterizations of invertibility. 

Let  $\mu \in \mathscr{E}'(\R^d)$ and $X_1, X_2 \subseteq \R^d$ be open sets such that \eqref{def-sets} holds. We set  $\check{\mu}(x) =  \mu(-x)$. Note that $S_{\check{\mu}} = S^t_{\mu} :  \mathscr{E}'(X_2) \rightarrow \mathscr{E}'(X_1)$. H\"ormander \cite[Theorem 16.5.7 and Corollary 16.5.19]{HoermanderPDO2} showed that $S_\mu : \mathscr{E}(X_1) \rightarrow \mathscr{E}(X_2)$ is surjective if and only if $\mu$ is invertible and the pair $(X_1,X_2)$ is $\mu$-convex for supports,  i.e.,  for every compact $K_1  \subseteq X_1$ there is a compact $K_2 \subseteq X_2$ such that for all $f \in \mathscr{E}'(X_2)$
\[\supp \check{\mu} \ast f \subseteq K_1 \Longrightarrow \supp f \subseteq K_2,\]
while $S_\mu : \mathscr{D}'(X_1) \rightarrow \mathscr{D}'(X_2)$ is surjective if and only if the pair $(X_1,X_2)$ is $\mu$-convex for supports and singular supports, where the latter means that  for every compact $K_1 \subseteq X_1$ there is a compact $K_2 \subseteq X_2$ such that for all $f \in \mathscr{E}'(X_2)$
\[\singsupp \check{\mu} \ast f \subseteq K_1 \Longrightarrow \singsupp f \subseteq K_2.\]
Furthermore, if $(X_1,X_2)$ is $\mu$-convex for singular supports, then $\mu$ is invertible \cite[Proposition 16.5.12]{HoermanderPDO2}. Hence, $S_\mu : \mathscr{E}(X_1) \rightarrow \mathscr{E}(X_2)$ is surjective if $S_\mu : \mathscr{D}'(X_1) \rightarrow \mathscr{D}'(X_2)$ is so.
By  \cite[Theorem 4.3.3]{HoermanderPDO1} and  \cite[Corollary 16.3.15]{HoermanderPDO2},  $\mu$ is invertible if and only if  $S_\mu : \mathscr{E}(\R^d) \rightarrow \mathscr{E}(\R^d)$ is surjective if and only if  $S_\mu : \mathscr{D}'(\R^d) \rightarrow \mathscr{D}'(\R^d)$ is surjective. This result was first shown by Ehrenpreis \cite[Theorem I]{Ehrenpreis}. We will use the above facts without explicitly referring to them.

Of course, every  differential operator $P(D)$, where $P \in \C[\xi_1, \ldots,  \xi_d]$ and $D = (-i\frac{\partial}{\partial x_1},\ldots,-i\frac{\partial}{\partial x_d})$, may be seen as a convolution operator with kernel $\mu = P(D) \delta$.

\section{Equivalence of ($\Omega$) and (P$\Omega$) for kernels of convolution operators}\label{sect-main}
The goal of this section is to show Theorem \ref{main-intro}. Fix   $\mu \in \mathscr{E}'(\R^d)$ and two open sets $X_1, X_2 \subseteq \R^d$ such that \eqref{def-sets} holds. We endow the spaces 
\[
\mathscr{E}_\mu(X_1,X_2)  = \{ f \in \mathscr{E}(X_1) \, | \, \mu \ast f =  0 \mbox{ in }   \mathscr{E}(X_2)  \}
\]
and 
\[
\mathscr{D}'_\mu(X_1,X_2) = \{ f \in \mathscr{D}'(X_1) \, | \, \mu \ast f =  0 \mbox{ in }   \mathscr{D}'(X_2) \}
\]
with the relative topology induced by  $\mathscr{E}(X_1)$ and  $\mathscr{D}'(X_1)$, respectively. Since both these spaces are closed, $\mathscr{E}_\mu (X_1,X_2)$ is a Fr\'echet space and $\mathscr{D}'_\mu (X_1,X_2)$ is a (PLS)-space.  

Let $(X_{1,n})_{n \in \N}$ and $(X_{2,n})_{n \in \N}$ be exhaustions by relatively compact open subsets of $X_1$ and $X_2$ such that  $X_{2,n} - \operatorname{supp} \mu \subseteq X_{1,n}$ for all $n \in \N$. Set $K_{j,n} = \overline{X_{j,n}}$ for $j = 1,2$ and note that $K_{2,n} - \operatorname{supp} \mu \subseteq K_{1,n}$ for all $n \in \N$. We define\footnote{Let $Y \subseteq \R^d$ be  relatively compact and open. $\mathscr{E}(\overline{Y})$ is the Fr\'echet space of smooth functions $f \in \mathscr{E}(Y)$ such that  $f^{(\alpha)}$ can be continuously extended to $\overline{Y}$ for all $\alpha \in \N^d$. $\mathscr{D}(\overline{Y})$ is the Fr\'echet space of smooth functionswith support in $\overline{Y}$ and $\mathscr{D}'(\overline{Y})$ is the strong dual of $\mathscr{D}(\overline{Y})$.}
\[
\mathscr{E}_\mu(K_{1,n} ,K_{2,n}) = \{  f \in \mathscr{E}(K_{1,n}) \, | \, \mu \ast f  = 0 \mbox{ in } \mathscr{E} (K_{2,n})  \}
\]
and
\[
\mathscr{D}'_\mu(K_{1,n} ,K_{2,n}) = \{  f \in \mathscr{D}'(K_{1,n}) \, | \, \mu \ast f  = 0 \mbox{ in } \mathscr{D}'(K_{2,n})  \}.
\]
Then, $(\mathscr{E}_\mu(K_{1,n} ,K_{2,n}) )_{n \in \N}$ is a projective spectrum of Fr\'echet spaces  such that 
\[
\mathscr{E}_\mu (X_1,X_2) = \operatorname{Proj} ( \mathscr{E}_\mu(K_{1,n} ,K_{2,n}) )_{n \in \N}
\] 
and $( \mathscr{D}'_\mu(K_{1,n} ,K_{2,n}))_{n \in \N}$ is a projective spectrum of (DFS)-spaces  such that 
\[
\mathscr{D}'_\mu (X_1,X_2) = \operatorname{Proj} ( \mathscr{D}'_\mu(K_{1,n} ,K_{2,n}))_{n \in \N}.
\]
In both cases we tacitly assumed that the spectral maps are the restriction maps. In the sequel we shall not write these maps.

For $n, N \in \N$ we define
\[
\| f \|_{n, N} = \max_{x \in  K_{1,n}, |\alpha| \leq N } |f^{(\alpha)}(x)|, \qquad f \in \mathscr{E}(K_{1,n}),
\]
and
\[
\| f \|^*_{n, N} = \sup \{ | \langle f, \varphi \rangle | \, | \, \varphi \in \mathscr{D}(K_{1,n}), \, \| \varphi \|_{n, N} \leq 1 \},  \qquad f \in \mathscr{D}'(K_{1,n}).
\]
We set
\[
U_{n,N} =  \{ f \in  \mathscr{E}_\mu(K_{1,n} ,K_{2,n}) \, | \, \| f \|_{n, N} \leq 1 \}, \qquad U_n = U_{n,n},
\]
and
\[
B_{n,N} = \{ f \in  \mathscr{D}'_\mu(K_{1,n} ,K_{2,n}) \, | \, \| f \|^*_{n, N} \leq 1 \}.
\]
Then, $(\frac{1}{N+1}U_{n,N})_{N \in \N}$ is a decreasing fundamental sequence of absolutely convex  neighborhoods of $0$ in $\mathscr{E}_\mu(K_{1,n} ,K_{2,n})$  and $(NB_{n,N})_{N \in \N}$ is an increasing fundamental sequence of absolutely convex bounded sets in $\mathscr{D}'_\mu(K_{1,n} ,K_{2,n})$. 
We need the following lemma.

\begin{lemma} \label{surj-proj} \mbox{}
	\begin{itemize}
		\item[$(i)$] The map  $S_\mu : \mathscr{E}(X_1) \rightarrow \mathscr{E}(X_2)$ is  surjective if and only if $\mu$ is invertible and 
		\[
		\operatorname{Proj}^1 ( \mathscr{E}_\mu(K_{1,n} ,K_{2,n}) )_{n \in \N} = 0.
		\]
		\item[$(ii)$] The map $S_\mu : \mathscr{D}'(X_1) \rightarrow \mathscr{D}'(X_2)$ is  surjective if and only if  $\mu$ is invertible and 
		\[
		\operatorname{Proj}^1 ( \mathscr{D}'_\mu(K_{1,n} ,K_{2,n}) )_{n \in \N} = 0.
		\] 
	\end{itemize}
\end{lemma}
\begin{proof}
	$(i)$  We start by showing that for invertible $\mu$ it holds that
	\begin{equation}
		\forall n \in \N \, \forall f \in \mathscr{E}(K_{2,n+1}) \, \exists g \in \mathscr{E}(\R^d) \, : \, \mu \ast g = f \mbox{ on } K_{2,n}.
		\label{invertible-con}
	\end{equation}
	Since $\mu$ is invertible,  there is $E \in\mathscr{D}'(\R^d)$ such that $\mu \ast E = \delta$. Choose $\psi \in \mathscr{D}(K_{2,n+1})$ such that $\psi = 1$ on $K_{2,n}$. Set $g = E \ast f\psi \in \mathscr{E}(\R^d)$ and note that $\mu \ast g = f\psi = f$ on $K_{2,n}$. The sufficiency of the condition for surjectivity is therefore a direct consequence of Proposition  \ref{proj1surj}. For its necessity we note that $\mu$ is invertible because $S_\mu : \mathscr{E}(X_1) \rightarrow \mathscr{E}(X_2)$ is  surjective, and thus that \eqref{invertible-con} holds. Furthermore,
	a standard cut-off and regularization argument shows that $(\mathscr{E}(K_{1,n}) )_{n \in \N}$ satisfies condition $(iii)$ from Theorem \ref{M-L}, whence $\operatorname{Proj}^1(\mathscr{E}(K_{1,n}) )_{n \in \N} =0$. Proposition  \ref{proj1surj} now implies that also $\operatorname{Proj}^1 ( \mathscr{E}_\mu(K_{1,n} ,K_{2,n}) )_{n \in \N} = 0$.
	
	$(ii)$ The proof is similar to the one of $(i)$ and is therefore omitted. We only remark that $\operatorname{Proj}^1(\mathscr{D}'(K_{1,n}) )_{n \in \N} =0$ since the restriction maps $\mathscr{D}'(K_{1,n+1}) \rightarrow \mathscr{D}'(K_{1,n})$ are surjective. 
\end{proof} 

Fix  $\chi \in \mathscr{D}(\R^d)$ with $\chi \geq 0$, $\operatorname{supp} \chi \subseteq B(0,1)$  and $\int_{\R^d} \chi(x)\dx =1$, and set $\chi_\varepsilon(x) = \varepsilon^{-d} \chi(x/\varepsilon)$ for $\varepsilon > 0$. We are ready to show  Theorem \ref{main-intro}.
\begin{proof}[Proof of Theorem \ref{main-intro}] $(i)$  By Lemma \ref{Omega-step} and a  rescaling argument, it suffices to show that
	\begin{eqnarray}\label{STS-omega}
		\begin{gathered}
			\forall n \in \N \, \exists m \geq n \, \forall k \geq m \, \exists C_1, C_2,s, \varepsilon_0 > 0 \, \forall \varepsilon \in (0, \varepsilon_0) \,:\\ 
			U_m \subseteq \varepsilon C_1  U_n + \frac{C_2}{\varepsilon^s} U_k.
		\end{gathered}
	\end{eqnarray}
	We start by showing that the spectrum $(\mathscr{D}'_\mu(K_{1,n} ,K_{2,n}) )_{n \in \N}$ is strongly reduced. We need to show that for all $n \in \N$ there is $m \geq n$ such that for all $U \in \mathscr{U}_0(\mathscr{D}'_\mu(K_{1,n} ,K_{2,n}))$ it holds
	\[ \mathscr{D}'_\mu(K_{1,m} ,K_{2,m})\subseteq \mathscr{D}'_\mu(X_{1} ,X_{2})+ U.\]
	As $S_\mu : \mathscr{E}(X_1) \rightarrow \mathscr{E}(X_2)$ is surjective, $\operatorname{Proj}^1 (\mathscr{E}_\mu(K_{1,n},K_{2,n} ))_{n \in \N} = 0$ by Lemma \ref{surj-proj}$(i)$. Let $n \in \N$ be arbitrary. Theorem \ref{M-L} yields that  there is  $\widetilde{m} \geq  n$ such that 
	\begin{equation}
		\mathscr{E}_\mu(K_{1,\widetilde{m} } ,K_{2,\widetilde{m}}) \subseteq \mathscr{E}_\mu(X_{1},X_2)+ U_{n,0}.\
		\label{M-L-reducedness}
	\end{equation}
	Set $m = \widetilde{m} + 1$ and let $U \in \mathscr{U}_0(\mathscr{D}'_\mu(K_{1,n} ,K_{2,n}))$ be arbitrary. Let $f \in \mathscr{D}'_\mu(K_{1,m} ,K_{2,m})$ be arbitrary. There is $\varepsilon > 0$  such that $f\ast \chi_\varepsilon \in \mathscr{E}_\mu(K_{1,\widetilde{m} } ,K_{2,\widetilde{m}})$ and $f - f\ast \chi_\varepsilon \in \frac{1}{2} U$. Next, choose  $\delta > 0$ such that $\delta U_{n,0} \subseteq \frac{1}{2} U$. By multiplying both sides of \eqref{M-L-reducedness} with $\delta$, we  obtain that
	\[
	f \ast \chi_\varepsilon \in  \mathscr{E}_\mu(K_{1,\widetilde{m} } ,K_{2,\widetilde{m}})  \subseteq \mathscr{E}_\mu(X_{1},X_2)+ \delta U_{n,0}\subseteq  \mathscr{D}'_\mu(X_{1} ,X_{2})   + \frac{1}{2} U.
	\]
	Hence,
	\[
	f = f \ast \chi_\varepsilon +  (f - f\ast \chi_\varepsilon) \in  \mathscr{D}'_\mu(X_{1} ,X_{2})  + \frac{1}{2} U + \frac{1}{2} U  =  \mathscr{D}'_\mu(X_{1} ,X_{2}) + U.
	\]
	This shows that the spectrum $(\mathscr{D}'_\mu(K_{1,n} ,K_{2,n}) )_{n \in \N}$ is strongly reduced, whence it satisfies (P$\Omega$).
	We now show \eqref{STS-omega}. Let $n \in \N$ be arbitrary. Choose $m \geq n+1$ according to  (P$\Omega$) and let $k \geq m$ be arbitrary. Condition (P$\Omega$) (applied to $k + 1$ and with $M = N$) and a rescaling argument yield that
	\[
	\begin{gathered}
		\exists N \in \N, \, K \geq N, \, C, r > 0 \, \forall \delta  \in (0,1) \, : |K_{1,m}| B_{m,N} \subseteq \delta B_{n+1,N} +  \frac{C}{\delta^{r}}B_{k+1,K},
	\end{gathered}
	\]
	where $|K_{1,m}|$ denotes the Lebesgue measure of $K_{1,m}$.
	Let $f \in U_m$ be arbitrary. Since $U_m \subseteq  |K_{1,m}|B_{m,N}$, there is $f_\delta \in \delta^{-r}CB_{k+1,K}$ with $f -f_\delta \in \delta B_{n+1,N}$  for all $\delta \in (0,1)$. Set $\varepsilon_0 = \min \{ 1, d(K_{1,k}, K^c_{1,k+1}),  d(K_{1,n}, K^c_{1,n+1}) \}$. For $\delta \in (0,1)$  and $\varepsilon \in (0, \varepsilon_0)$ we define $f_{\delta, \varepsilon} =  f_\delta \ast \chi_\varepsilon \in \mathscr{E}_\mu(K_{1,k},K_{2,k})$. The mean value theorem implies that
	\[
	f - f \ast \chi_\varepsilon \in  \varepsilon \sqrt{d}   U_n.
	\]
	Since $f -f_\delta \in \delta B_{n+1,N}$, we have that
	\[
	f \ast \chi_\varepsilon - f_{\delta, \varepsilon} = (f - f_\delta) \ast \chi_\varepsilon \in  \|\chi\|_{\infty, N + n} 
	\frac{\delta}{\varepsilon^{N+n+d}}  U_n,
	\]
	where $\|\chi\|_{\infty,j}=\max_{x\in\R^d,|\alpha|\leq j}|\chi^{(\alpha)}(x)|$, $j\in\N$. Similarly, as $f_\delta \in \delta^{-r}CB_{k+1,K}$, it holds that
	\[
	f_{\delta, \varepsilon}  \in  C\|\chi\|_{\infty, K + k}   \frac{1}{\delta^r \varepsilon^{K+k+d}}U_k.
	\]
	Define $f_\varepsilon = f_{\varepsilon^{N+n+d+1}, \varepsilon} \in \mathscr{E}_\mu(K_{1,k},K_{2,k})$. Set $C_1 =  \sqrt{d} +  \|\chi\|_{\infty, N + n}$,  $C_2 = C  \|\chi\|_{\infty, K + k}$ and $s = r(N+n+d+1) + (K + k +d)$. The above estimates yield that
	\[
	f =  (f - f \ast \chi_\varepsilon) + (f \ast \chi_\varepsilon - f_{\varepsilon}) + f_\varepsilon \in \varepsilon C_1  U_n + \frac{C_2}{\varepsilon^s} U_k.
	\]
	
	$(ii)$  Lemma \ref{surj-proj}$(ii)$  yields that $\operatorname{Proj}^1 (\mathscr{D}'_\mu(K_{1,n},K_{2,n}) )_{n \in \N} = 0$. By \cite[Theorem 3.2.9]{Wengenroth}, we obtain that $(\mathscr{D}'_\mu(K_{1,n},K_{2,n}) )_{n \in \N}$ is strongly reduced. Hence, it suffices to prove that $(\mathscr{D}'_\mu(K_{1,n},K_{2,n}) )_{n \in \N}$ satisfies (P$\Omega$). By a  rescaling argument, it is enough to show that 
	\begin{eqnarray}\label{STS-Pomega}
		\begin{gathered}
			\forall n \in \N \, \exists m \geq n \, \forall k \geq m \, \exists N \in \N \, \forall M \geq N \, \exists K \geq N, C_1,C_2,s, \varepsilon_0 > 0 \\ \forall  \varepsilon \in (0, \varepsilon_0) \, : 
			\, 
			B_{m,M} \subseteq  \varepsilon C_1  B_{n,N} + \frac{C_2}{\varepsilon^s} B_{k,K}.
		\end{gathered}	
	\end{eqnarray}
	$S_\mu : \mathscr{E}(X_1) \rightarrow \mathscr{E}(X_2)$ is surjective because $S_\mu : \mathscr{D}'(X_1) \rightarrow \mathscr{D}'(X_2)$ is so. Hence, Lemma \ref{surj-proj}$(i)$ implies that $\operatorname{Proj}^1 (\mathscr{E}_\mu(K_{1,n},K_{2,n} ))_{n \in \N} = 0$. Therefore, by Lemma  \ref{Omega-step}, $(\mathscr{E}_\mu(K_{1,n},K_{2,n} ))_{n \in \N} $  satisfies  \eqref{Omega-step-form}. As we have that $\operatorname{Proj}^1 (\mathscr{D}'_\mu(K_{1,n},K_{2,n}) )_{n \in \N} = 0$, Theorem \ref{M-L-DFS} yields that
	\begin{eqnarray}\label{proj12}
		\begin{gathered}
			\forall n \in \N \, \exists m \geq n \, \forall k \geq m \, \exists N \in \N \, \forall M \geq N\, \exists K \geq N, C > 0\,  :  \\
			B_{m,M} \subseteq  B_{n,N} +  CB_{k,K}.
		\end{gathered}
	\end{eqnarray}
	We now show  \eqref{STS-Pomega}. Let $n \in \N$ be arbitrary. Choose $m_1 \geq n$ according to  \eqref{Omega-step-form} and $m_2 \geq n$ according to \eqref{proj12}. Set $m = \max\{m_1, m_2\}  +1$. Let $k \geq m$ be arbitrary. Condition  \eqref{Omega-step-form} implies that there are  $C,r >0$ such that for all $\delta \in (0,1)$
	\begin{equation}
		\label{Omega1}
		U_{m_1} \subseteq \delta U_{n,0} + \frac{C}{\delta^r} U_{k,0}.
	\end{equation}
	Choose $N \in \N$ according \eqref{proj12}. Let $M \geq N$ be arbitrary. Condition \eqref{proj12} (applied to $M+1$ instead of $M$) implies that there are $K \geq M+1$ and $C' > 0$ such that
	\begin{equation}
		\label{proj13}
		B_{m_2,M+1} \subseteq  B_{n,N} +  C'B_{k,K}.
	\end{equation}
	Set $\varepsilon_0 = \min \{ 1, d(K_{1,m-1}, K^c_{1,m}) \}$. Let $f \in  B_{m,M}$ be arbitrary. For $\varepsilon \in (0, \varepsilon_0)$ we define $f_\varepsilon = f \ast \chi_\varepsilon \in \mathscr{E}_\mu(K_{1,m-1}, K_{2,m-1})$.  The mean value theorem and \eqref{proj13} imply that
	\[
	f- f_\varepsilon \in \varepsilon \sqrt{d} B_{m_2,M+1} \subseteq \varepsilon  \sqrt{d}B_{n,N} + \sqrt{d}C' B_{k,K}.
	\]
	We  have that
	\[
	f_\varepsilon \in \| \chi\|_{\infty, M+m_1}\frac{1}{\varepsilon^{M+m_1+d}} U_{m_1}. 
	\]
	Set $s =r(M+m_1+d+1) + M+m_1+d$. By \eqref{Omega1} (applied to $\delta = \varepsilon^{M+m_1+d+1}$), we obtain that
	\[
	\frac{1}{\| \chi\|_{\infty, M+m_1}} f_\varepsilon \in \frac{1}{\varepsilon^{M+m_1+d}} U_{m_1} \subseteq \varepsilon U_{n,0} +  \frac{C}{\varepsilon^s} U_{k,0} \subseteq   \varepsilon |K_{1,n}| B_{n,N} +  \frac{C}{\varepsilon^s} |K_{1,k}| B_{k,K}.
	\]
	Set $C_1 = \sqrt{d} +|K_{1,n}| \| \chi\|_{\infty, M+m_1}$ and $C_2 = \sqrt{d}C' +|K_{1,k}| \| \chi\|_{\infty, M+m_1}C$.  The above inclusions yield that
	\[
	f \in C_1\varepsilon  B_{n,N} +\frac{C_2}{\varepsilon^s} B_{k,K}.
	\]
\end{proof}
Recall from the introduction that  for a  differential operator $P(D)$ and an open set $X \subseteq \R^d$ we simply write
\[
\mathscr{E}_P(X)  = \{ f \in \mathscr{E}(X) \, | \, P(D)f = 0 \}
\]
and
\[
\mathscr{D}'_P(X)  = \{ f \in \mathscr{D}'(X) \, | \, P(D)f = 0 \}.
\]

\begin{remark} Theorem \ref{main-intro} with ($\Omega$) replaced by $(\overline{\overline{\Omega}})$ \cite{Vogt1983}  and  (P$\Omega$) replaced by $(P\overline{\overline{\Omega}})$ \cite{B-D2006}  does not hold. In fact, Vogt \cite[Theorem 14]{Vogt2006} showed that for any open convex set $X \subseteq \R^d$, $d >1$, and any $P \in \C[\xi_1, \ldots, \xi_d]$, the space $\mathscr{E}_P(X)$ does not satisfy $(\overline{\overline{\Omega}})$. On the contrary, as $\mathscr{D}'(X)$ satisfies $(P\overline{\overline{\Omega}})$ \cite[Corollary 6.1]{B-D2006} and this condition is inherited by complemented subspaces, $\mathscr{D}'_P(X)$ satisfies $(P\overline{\overline{\Omega}})$ for any open set  $X \subseteq \R^d$ and any $P(D)$ such that $P(D): \mathscr{D}'(X) \rightarrow \mathscr{D}'(X)$ admits a continuous linear right inverse \cite{MTV}. 
\end{remark}

\section{The condition ($\Omega$) for smooth kernels of convolution and differential operators}\label{sect-examples}
\subsection{The augmented operator} Let $X_1, X_2 \subseteq \R^d$ be open and let $T: \mathscr{D}'(X_1) \rightarrow \mathscr{D}'(X_2)$ be a surjective continuous linear map.  Bonet and Doma\'nski \cite[Proposition 8.3]{B-D2006} showed that $\ker T$ satisfies $(P \Omega)$ if and only if
\[
T \otimes \operatorname{id}_{\mathscr{D}'(\R)} :  \mathscr{D}'(X_1 \times \R) \rightarrow \mathscr{D}'(X_2 \times \R)
\]
is surjective\footnote{In fact, they  only showed this statement for $X_1= X_2$. Since $\mathscr{D}'(X_1) \cong (s')^{\N}\cong \mathscr{D}'(X_2) $ for any two open sets $X_1, X_2 \subseteq \R^d$, the above more general result is a consequence  of the particular case $X_1= X_2$. Alternatively, the proof of \cite[Proposition 8.3]{B-D2006} can  be readily adapted. }.  

Let  $\mu \in \mathscr{E}'(\R^d)$ and $X_1, X_2 \subseteq \R^d$ be open sets such that \eqref{def-sets} holds.  Suppose that $S_\mu: \mathscr{D}'(X_1) \rightarrow \mathscr{D}'(X_2)$ is surjective. By the above result for $T =S_\mu $, we find that $\mathscr{D}'_\mu(X_1,X_2)$ satisfies $(P \Omega)$ if and only if 
\[
S_\mu  \otimes \operatorname{id}_{\mathscr{D}'(\R)} = S_{\mu \otimes \delta} : \mathscr{D}'(X_1 \times \R) \rightarrow \mathscr{D}'(X_2 \times \R)
\]
is surjective. We call $ S_{\mu \otimes \delta}$ the \emph{augmented operator} of $S_\mu$. In particular, for a  differential operator $P(D)$ and an open set $X \subseteq \R^d$, the augmented operator of $P(D): \mathscr{D}'(X) \rightarrow \mathscr{D}'(X)$ is given by the  differential operator $P^+(D): \mathscr{D}'(X \times \R) \rightarrow \mathscr{D}'(X \times \R)$, where $P^+(\xi_1, \ldots, \xi_{d+1}) = P(\xi_1, \ldots, \xi_d)$ (cf.\ the introduction). 

Theorem \ref{main-intro} enables us to also relate the condition ($\Omega$) for the space of smooth zero solutions of a convolution operator to the surjectivity of its augmented operator -- as explained in the introduction, before this was only known for hypoelliptic differential operators. More precisely, the following result holds.
\begin{theorem}\label{main-1}
	Let  $\mu \in \mathscr{E}'(\R^d)$ and $X_1, X_2 \subseteq \R^d$ be open sets such that \eqref{def-sets} holds. Suppose that $S_\mu: \mathscr{D}'(X_1) \rightarrow \mathscr{D}'(X_2)$ is surjective. The following statements are equivalent:
	\begin{itemize}
		\item[$(i)$] $\mathscr{E}_\mu(X_1,X_2)$ satisfies ($\Omega$).
		\item[$(ii)$] $\mathscr{D}'_\mu(X_1,X_2)$ satisfies (P$\Omega$).
		\item[$(iii)$] $S_{\mu \otimes \delta}: \mathscr{D}'(X_1 \times \R) \rightarrow \mathscr{D}'(X_2 \times \R)$ is surjective.
		\item[$(iv)$] $(X_1 \times \R,X_2 \times \R)$ is $\mu \otimes \delta$-convex for singular supports.
	\end{itemize}
\end{theorem}
\begin{proof}
	$S_\mu: \mathscr{E}(X_1) \rightarrow \mathscr{E}(X_2)$ is surjective as $S_\mu: \mathscr{D}'(X_1) \rightarrow \mathscr{D}'(X_2)$ is so.
	
	\noindent $(i) \Leftrightarrow (ii)$ This is shown in Theorem \ref{main-intro}. \\
	\noindent $(ii) \Leftrightarrow (iii)$ As explained above, this follows from  \cite[Proposition 8.3]{B-D2006}. \\
	\noindent $(iii) \Leftrightarrow (iv)$ $S_{\mu \otimes \delta}: \mathscr{D}'(X_1 \times \R) \rightarrow \mathscr{D}'(X_2 \times \R)$ is surjective if and only if  $(X_1 \times \R,X_2 \times \R)$ is $\mu \otimes \delta$-convex for supports and  singular supports. Hence, $(iii) \Rightarrow (iv)$ is trivial. For the converse direction, it suffices to note that $(X_1 \times \R,X_2 \times \R)$ is $\mu \otimes \delta$-convex for supports since $(X_1,X_2)$ is $\mu$-convex for supports (cf.\ \cite[Proposition 1]{FrKa}).
\end{proof}

In the next two subsections, we will use Theorem \ref{main-1} to show ($\Omega$)  for the space of smooth zero solutions of several types of convolution and differential operators.

\subsection{Differential operators}\label{PDO}
In \cite[Corollary 8.4]{B-D2006} it is shown that $\mathscr{D}'_P(X)$  satisfies (P$\Omega$) for any differential operator $P(D)$ and any  open convex set $X \subseteq \R^d$. In view of Theorem \ref{main-1}, this follows from the fact that  $P^+(D): \mathscr{D}'(X \times \R) \rightarrow \mathscr{D}'(X \times \R)$ is surjective (as  $X \times \R$ is convex).  Hence, Theorem \ref{main-1} yields the following counterpart of this result in the smooth setting. 
\begin{theorem}\label{theo: convex}
	Let $P\in\C[\xi_1,\ldots,\xi_d]$ and let $X \subseteq\R^d$ be open and convex. Then, $\mathscr{E}_P(X)$ satisfies ($\Omega$).
\end{theorem}

\begin{remark} $(i)$ Theorem \ref{theo: convex} has been shown by Petzsche in \cite[Corollary 4.5]{Petzsche1980} under the additional hypothesis that $P$ is hypoelliptic.  A careful inspection of his proof, which is based on the fundamental principle of  Ehrenpreis, actually shows that his method may be used to prove  Theorem \ref{theo: convex} in its full generality. 
	
	\noindent $(ii)$  We will extend Theorem \ref{theo: convex} to  difference-differential operators in Corollary \ref{diff-diff} below. 
\end{remark}

Next, we combine Theorem \ref{main-1} with several results of the second author \cite{Kalmes12-3,Kalmes19} about the surjectivity of  augmented operators to show ($\Omega$) for spaces of  smooth zero solutions of certain (non-hypoelliptic) differential operators. We say that a polynomial $P\in\C[\xi_1,\ldots,\xi_d]$  \emph{acts along a subspace} $W\subseteq\R^d$ if  $P(x)=P(\pi_Wx)$ for all  $x\in\R^d$, where $\pi_W$ denotes the orthogonal projection onto $W$. A polynomial $P$ which acts along a subspace $W$ is said to be \emph{elliptic on} $W$ if for its principal part $P_m$ it holds that $P_m(x)\neq 0$ for every $x\in W\backslash\{0\}$.

\begin{theorem}\label{theo-main-examples}
	Let $P\in\C[\xi_1,\ldots,\xi_d]$ and $X \subseteq \R^d$ be open such that $P(D):\mathscr{E}(X)\rightarrow\mathscr{E}(X)$ is surjective. Then, $\mathscr{E}_P(X)$ satisfies ($\Omega$) in the following cases:
	\begin{itemize}
		\item[$(i)$] $P$ acts along a subspace of $\R^d$ and is elliptic there.
		\item[$(ii)$] $P(D)$ is a product of first order operators, i.e., with  $\alpha\in\C\backslash\{0\}$, $c_j\in\C$ and $N_j\in\C^d\backslash\{0\}$, $j=1,\ldots,l$ it holds that $P(x)=\alpha\prod_{j=1}^l\left(N_j \cdot x-c_j\right)$.
		\item[$(iii)$] $d = 2$.
	\end{itemize}
\end{theorem}
\begin{proof}
	In view of Theorem \ref{main-intro}, this follows from the fact that in all three cases the surjectivity of $P(D):\mathscr{E}(X)\rightarrow\mathscr{E}(X)$ implies that $P(D):\mathscr{D}'(X)\rightarrow\mathscr{D}'(X)$ is surjective and  that $\mathscr{D}'_P(X)$ satisfies (P$\Omega$) \cite[Theorem 9 and Theorem 18(b)]{Kalmes19} for $(i)$; \cite[Corollary 10 and Theorem 18(c)]{Kalmes19} for $(ii)$;  \cite[Theorem 21]{Kalmes12-3} for $(iii)$.
\end{proof}

\begin{remark}
	In all three cases considered in Theorem \ref{theo-main-examples} a geometric characterization of the sets $X$ for which  $P(D):\mathscr{E}(X)\rightarrow\mathscr{E}(X)$ is surjective is known. A function $f:X\rightarrow [0,\infty]$ is said to satisfy the \emph{minimum principle} in a closed set $F\subseteq\R^d$ if for every compact subset $K$ of $X \cap F$ it holds that
	\[\inf_{x\in K}f(x)=\inf_{x\in\partial_F K}f(x),\]
	where $\partial_F K$ denotes the boundary of $K$ in $F$. Then,  $P(D):\mathscr{E}(X)\rightarrow\mathscr{E}(X)$ is surjective if and only if 
	\begin{itemize}
		\item[$(i)$] \emph{($P$ acts along a subspace $W$ of $\R^d$ and is elliptic there)} The boundary distance $d_X: X \to \R$ satisfies the minimum principle in the affine subspace $x+W$ for every $x\in\R^d$.
		\item[$(ii)$] \emph{($P(x)=\alpha\prod_{j=1}^l\left(N_j \cdot x-c_j\right)$ for some $\alpha\in\C\backslash\{0\}$, $c_j\in\C$ and $N_j\in\C^d\backslash\{0\}$, $j=1,\ldots,l$)} The boundary distance $d_X: X \to \R$  satisfies the minimum principle in the affine subspace $x+\mbox{span}\{\mbox{Re}N_j,\mbox{Im}N_j\}$ for every $x\in\R^d$ and  $j=1,\ldots,l$.
		\item[$(iii)$] \emph{($d = 2$)} The intersection of every characteristic line of $P$ with any connected component of $X$ is an interval.
	\end{itemize}
	$(i)$ and $(ii)$  are shown in \cite[Theorem 9 and Corollary 10]{Kalmes19}, while $(iii)$ is a classical result of H\"ormander \cite[Theorem 10.8.3]{HoermanderPDO2}.
\end{remark}

\subsection{Convolution operators - the convex case}\label{CO}
In \cite[Corollary 8.5]{B-D2006} it is shown that $\mathscr{D}'_\mu(\R^d,\R^d)$ satisfies (P$\Omega$) for any invertible $\mu \in \mathscr{E}'(\R^d)$. This follows from  Theorem \ref{main-1}: $\mu \otimes \delta$ is invertible because $\mu$ is so, whence $S_{\mu \otimes \delta}: \mathscr{D}'(\R^{d+1}) \rightarrow \mathscr{D}'(\R^{d+1})$ is surjective. Theorem \ref{main-1} therefore yields the following counterpart of this result in the smooth setting. 
\begin{theorem}\label{theo: convex-conv}
	Let $\mu \in \mathscr{E}'(\R^d)$ be invertible. Then, $\mathscr{E}_\mu(\R^d,\R^d)$ satisfies ($\Omega$).
\end{theorem}

In the rest of this section we extend \cite[Corollary 8.5]{B-D2006}  and Theorem \ref{theo: convex-conv} to arbitrary convex sets for a certain class of convolution operators. Let  $\mu \in \mathscr{E}'(\R^d)$.  For a convex open set $X \subseteq \R^d$ we set 
\[
X_\mu = \{ x \in \R^d \, | \, x - \supp \mu \subseteq X \}.
\]
Let $Y \subseteq \R^d$ be an open convex set such that $Y - \supp \mu \subseteq X$. By \cite[Example 16.5.5 and Proposition 16.5.6]{HoermanderPDO2}, the pair $(X, Y)$ is $\mu$-convex for supports if and only if $Y = X_\mu$. Hence,  if $\mu$ is invertible, $S_{\mu}: \mathscr{E}(X) \rightarrow \mathscr{E}(Y)$ is surjective if and only if $Y = X_\mu$. On the contrary, there exist invertible $\mu \in \mathscr{E}'(\R^d)$ and  open convex sets $X \subseteq \R^d$ such that $(X, X_\mu)$ is not $\mu$-convex for singular supports, or equivalently, that $S_{\mu}: \mathscr{D}'(X) \rightarrow \mathscr{D}'(X_\mu)$ is not surjective. For example, let $\mu$ be the area element of the boundary of an ellipsoid in $\R^d$ and let $X$ be an open half-space. Then, \cite[Theorem 16.3.20 and Proposition 16.5.14]{HoermanderPDO2} imply that  $(X, X_\mu)$ is not $\mu$-convex for singular supports. Consequently, the simple argument used to show Theorem \ref{theo: convex} cannot be extended to general convolution operators. However, we now show that this can be done for  convolution operators whose kernels are distributions of  class $\mathscr{R}$. This class of distribution kernels was introduced and  studied by Berenstein and Dostal \cite{Be-Do1973,Be-Do1973-1,Dostal}. A distribution $\mu \in \mathscr{E}'(\R^d)$ is said to be of class  $\mathscr{R}$ \cite[Definition 1]{Be-Do1973-1} if 
\[
\exists c,R,M > 0 \, \forall \zeta \in \C^d \, \exists z \in \C^d, |\zeta - z| < R \, : \,  |\widehat{\mu}(z)| \geq c(1+|\operatorname{Re} \zeta|)^{-M} e^{H_{\mu}(\operatorname{Im} \zeta) }.
\]
We recall that  $H_\mu$ denotes the supporting function of $ \supp \mu$. The following examples and properties of distributions of class $\mathscr{R}$ are taken from \cite{Be-Do1973-1}:

\begin{example}\label{examples-R} \mbox{}
	\begin{itemize}
		\item[$(i)$] Every distribution with finite support is of  class  $\mathscr{R}$.
		\item[$(ii)$] The characteristic function  of a compact polyhedron and the area element of its boundary are of  class  $\mathscr{R}$.
		\item[$(iii)$] Let $\mu, \nu \in \mathscr{E}'(\R^d)$. Then, $\mu \ast \nu$ is of class $\mathscr{R}$ if and only if $\mu, \nu$ are of  class  $\mathscr{R}$.
		\item[$(iv)$] The characteristic function of an ellipsoid and the area element of its boundary are not of class  $\mathscr{R}$.
	\end{itemize}
\end{example}

Our interest in distributions of class $\mathscr{R}$ stems from the following result.
\begin{proposition}\label{surj-R} \cite[Proposition 1 and Proposition 2]{Be-Do1973-1}
	Let $\mu \in \mathscr{E}'(\R^d)$ be of class  $\mathscr{R}$. For every  open convex set $X \subseteq \R^d$ it holds that $S_{\mu}: \mathscr{D}'(X) \rightarrow \mathscr{D}'(X_\mu)$ is surjective.
\end{proposition}
Proposition \ref{surj-R} enables us to show the following result. 
\begin{theorem}\label{theo:conv-conv}
	Let $\mu \in \mathscr{E}'(\R^d)$ be of class  $\mathscr{R}$. For every  open convex set $X \subseteq \R^d$ it holds that  $\mathscr{D}'_\mu(X,X_\mu)$ satisfies (P$\Omega$) and  $\mathscr{E}_\mu(X,X_\mu)$ satisfies ($\Omega$).
\end{theorem}
\begin{proof}
	Note that $H_{\mu \otimes \delta}(\xi_1, \ldots, \xi_{d+1}) = H_{\mu}(\xi_1, \ldots, \xi_{d})$. Hence,  $\mu \otimes \delta$  is of class  $\mathscr{R}$ because $\mu$ is so. The result therefore follows from Theorem \ref{main-1} and Proposition \ref{surj-R}.
\end{proof}

As difference-differential operators are precisely the convolution operators whose kernels have finite support,  Example \ref{examples-R}$(i)$ and Theorem \ref{theo:conv-conv} yield the next result.

\begin{corollary}\label{diff-diff}
	Let $a_1, \ldots, a_l \in \R^d$ and $P_1, \ldots, P_l \in \C[\xi_1, \ldots, \xi_d]$. Let $X \subseteq \R^d$ be convex and set 
	$Y = \bigcap_{k =1}^l( a_k + X)$. Consider the difference-differential operator 
	\[
	S : \mathscr{D}'(X) \to  \mathscr{D}'(Y), \quad  S(f)(x) = \sum_{k = 1}^l P_k(D)f(x - a_k). 
	\]
	Then, $\ker (S: \mathscr{D}'(X) \to \mathscr{D}'(Y))$  satisfies (P$\Omega$) and  $\ker (S: \mathscr{E}(X) \to \mathscr{E}(Y))$  satisfies ($\Omega$).
\end{corollary}

We end this article by posing the following problem:
\begin{problem}\label{problem-1}
	Let  $\mu \in \mathscr{E}'(\R^d)$ be invertible  and $X \subseteq \R^d$ be open and convex. \\
	$(i)$  Suppose that $S_\mu : \mathscr{D}'(X) \rightarrow \mathscr{D}'(X_\mu)$ is surjective. Does $\mathscr{D}'_\mu(X,X_\mu)$ satisfy (P$\Omega$)?  \\
	$(ii)$ Does $\mathscr{E}_\mu(X,X_\mu)$ satisfy ($\Omega$)? Recall that  $S_\mu : \mathscr{E}(X) \rightarrow \mathscr{E}(X_\mu)$ is always surjective.
\end{problem}
\noindent By using the theory of plurisubharmonic functions,  H\"ormander  \cite[Proposition 16.5.14]{HoermanderPDO2} gave a geometrical characterization of $\mu$-convexity for singular supports for pairs of  open convex sets in terms of the family $\mathscr{H}(\mu)$  \cite[Definition 16.3.2]{HoermanderPDO2}. It would be interesting to investigate if this characterization could be used to tackle Problem \ref{problem-1}$(i)$. More precisely, one could try to show condition $(iv)$ from Theorem \ref{main-1} by verifying the condition in H\"ormander's result for $\mu \otimes \delta$. To this end, one would need a good description of  $\mathscr{H}(\mu \otimes \delta)$ in terms of  $\mathscr{H}(\mu)$. It is unclear to the authors how to obtain this.

Incidentally, for distributions $\mu$ with finite support  it holds that $\supp\mu=\singsupp\mu$ and $\mathscr{H}(\mu) = \{ H_\mu \}$ as well as $\mathscr{H}(\mu \otimes \delta) = \{ H_{\mu\otimes\delta} \}$ (see the remark after \cite[Corollary 16.3.18]{HoermanderPDO2}). Hence, $\mu$ is invertible   \cite[Theorem 16.3.10]{HoermanderPDO2}  and $(X,X_\mu)$ is $\mu$-convex for singular supports and $(X\times\R,X_\mu\times\R)$ is $\mu\otimes\delta$-convex for singular supports  \cite[Corollary 16.5.15]{HoermanderPDO2}. In view of Theorem \ref{main-1}, this gives an alternative proof of Corollary \ref{diff-diff} (recall that $(X,X_\mu)$ is always $\mu$-convex for supports). 


\end{document}